\documentclass[11pt,intlimits]{amsart}
\usepackage{amsthm,amsfonts,amssymb,euscript, layout,setspace}

\newcommand{\R}{\ensuremath{\mathbb{R}}}
\newcommand{\Z}{\ensuremath{\mathbb{Z}}}

\newcommand{\Ss}{\ensuremath{\mathcal{S}}}

\newcommand{\BB}{\ensuremath{\mathcal{B}}}

\newcommand{\x}{\ensuremath{\mathbf{x}}}
\newcommand{\y}{\mathbf{y}}

\newcommand{\ep}{\ensuremath{\epsilon}}
\newcommand{\la}{\ensuremath{\lambda}}
\newcommand{\La}{\ensuremath{\Lambda}}
\newcommand{\te}{\ensuremath{\theta}}
\newcommand{\al}{\ensuremath{\alpha}}

\newcommand{\ga}{\ensuremath{\gamma}}
\newcommand{\de}{\ensuremath{\delta}}

\newcommand{\Om}{\ensuremath{\Omega}}

\newcommand{\lap}{\ensuremath{\Delta}}
\newcommand{\st}{\ensuremath{\; | \;}}

\newcommand{\closure}[1]{\overline{#1}}
\newcommand{\weaklyto}{\rightharpoonup}
\newcommand{\abs}[1]{|#1|}

\newcommand{\D}[2]{\ensuremath{\frac{\partial #1}{\partial #2}}}
\newcommand{\DD}[3]{\ensuremath{\frac{\partial^2 #1}{\partial #2 \partial #3}}}
\newcommand{\Dd}[2]{\ensuremath{\frac{\partial^2 #1}{\partial^2 #2}}}

\newcommand{\Pe}{P($\epsilon$)}
\newcommand{\Pee}{P($\epsilon'$)}

\DeclareMathOperator{\divg}{div}
\DeclareMathOperator{\dist}{dist}

\newtheorem{theorem}{Theorem}[section]
\newtheorem{corollary}[theorem]{Corollary}
\newtheorem{lemma}[theorem]{Lemma}

\theoremstyle{definition}

\theoremstyle{remark}
\newtheorem{remark}{Remark}[section]
\numberwithin{equation}{section}

\begin{document}

\title[A free-boundary problem for the evolution $p$-Laplacian]
{A free-boundary problem for the evolution $p$-Laplacian equation with a combustion boundary condition}

\author{Tung To}
\address{Department of Mathematics\\
University of Chicago\\
Chicago, IL 60637, USA}
\email{totung@math.uchicago.edu}


\begin{abstract}
We study the existence, uniqueness and regularity of solutions of the equation $f_t = \lap_p f = \divg\,(|Df|^{p-2}\,Df)$ under over-determined boundary conditions $f = 0$ and $|Df| = 1$. We show that if the initial data is concave and Lipschitz with a bounded and convex support, then the problem admits a unique solution which exists until it vanishes identically. Furthermore, the free-boundary of the support of $f$ is smooth for all positive time.
\end{abstract}

\maketitle
\section{Introduction}\label{s-introduction}
Fix a number $p > 2$. Given a non-negative function $f_0$ on $\R^n$ with positive set $\Om_0$, we want to find a non-negative function $f(x,t)$ on $\R^n \times (0,T)$ with positive set $\Om$ which solves the following problem:
\begin{equation}\label{P}\tag{P}
\begin{cases}
f_t = \lap_p f &\text{in $\Om = \{ f > 0 \}$}\\
f = 0 \text{ and } |Df| = 1 &\text{on $\partial\Om \cap \{ 0 < t < T \}$}\\
\lim_{t \to 0} f(\x,t) = f_0(\x) &\forall\, \x \in \R^n.
\end{cases}
\end{equation}

The operator
\[
\lap_p f = \divg\, ( |Df|^{p-2}\,Df )
\]
is known as the $p$-Laplacian. In non-divergent form, it can be written as
\begin{equation}
\lap_p = |Df|^{p-2} \lap f + (p-2)|Df|^{p-4} f_{ij} f_i f_j
\end{equation}
Note that the Einstein summation notation was used in the last term. It can also be written as
\begin{equation}
\lap_p = |Df|^{p-2} ( \lap f + (p-2)f_{\nu\nu} )
\end{equation}
where $f_{\nu\nu}$ denotes the second derivative of $f$ in the direction of $\nu = Df/|Df|$.

In the case $p > 2$, this operator is nonlinear and degenerate at vanishing points of $Df$. When $p = 2$, it is just the regular Laplacian.

Due to the over-determined boundary conditions $f = 0$ and $|Df| = 1$, the time-section of $\Om$
\[
\Om_t = \{ \x \in \R^n \st f(\x,t) > 0 \}
\]
will in general change with time. In other words, the boundary $\partial\Om_t$ moves. It is often known as the moving-boundary or free-boundary.

Our work is motivated by the work of Caffarelli and V\'azquez \cite{CaffarelliV1995} in which authors studied this problem in the case $p=2$. Their result stated essentially that if $\partial\Om_0 \in C^2$, $f_0 \in C^2(\Om_0)$ and $\lap f_0 \leq 0$, then there exists a solution to the problem. Moreover, if $\Om_0$ is compact, solutions vanish in finite time. Still in this case, long-time existence, uniqueness and regularity of the free-boundary have been studied by Daskalopoulos and Ki-Ahm Lee \cite{DaskalopoulosL2002} or Petrosyan \cite{Petrosyan2001, Petrosyan2002} when the initial value is concave or star-shaped with bounded support. Other kinds of solution have also been studied (see also \cite{LedermanVW2001}). In the case $p > 2$, an elliptic version of the problem has been studied before by Danielli, Petrosyan and Shahgholian \cite{DanielliPS2003} or Henrot and Shahgholian \cite{HenrotS2000a, HenrotS2000}. As far as the parabolic problem when $p > 2$ is concerned, the only result we are aware of is by Akopyan and Shahgholian \cite{AkopyanS2001} where authors showed the uniqueness under the hypotheses that the time-section $\Om_t$ is convex and non-decreasing in time. The questions of existence or regularity of the free-boundary were not addressed in that paper.

The main result of our work is stated below.

\begin{theorem}
Assume that $\Om_0$ is a bounded and convex domain. The function $f_0$ is positive and concave in $\Om_0$. Furthermore, on the boundary $\partial\Om_0$, $f_0$  satisfies
\begin{align*}
f_0(\x) &= 0 \quad\text{for all $\x$}\\
|Df_0(\x)| &= 1 \quad\text{for a.e. $\x$}.
\end{align*}
Then the problem \eqref{P} has a unique solution up to a finite time $T$ where it vanishes identically in the sense that
\[
\lim_{t \to T} f(\x,t) = 0 \quad \forall\; \x \in \R^n.
\]
Moreover, the free-boundary $\partial\Om_t$ is smooth for all $t \in (0,T)$.
\end{theorem}

It is well-known that solutions of the evolution $p$-Laplacian are only $C^{1,\al}$ at points of vanishing gradient (see for example \cite{DiBenedetto1993}). Hence, solutions to the problem \eqref{P} must be defined in some weak sense. We will state precisely the meaning of our solution in section \ref{s-statement}.

Our approach to the problem is totally different from \cite{CaffarelliV1995}. To deal with the degeneracy, we will approximate the $p$-Laplacian with the following regularized operator
\begin{equation*}\tag{\Pe}\label{Pe}
\lap_p^\ep f = \divg ( (|Df|^2 + \ep)^{q - 1}\,Df ).
\end{equation*}
Here and throughout this work, we define $q = p/2$. We will establish some properties for solutions of these regularized problems and then let $\ep$ go to 0 to obtain a solution to the degenerate problem.

In order to solve this regularized free-boundary problem, we employ a change of coordinates that transforms it into a quasilinear equation with Neumann boundary condition on a fixed-domain problem. Applying results from standard theory of quasi-linear parabolic equations with oblique boundary condition, we show that this new problem admits a solution for some positive time. Revert back to the original coordinates, we obtain a short-time existence result for the regularized problem. This argument is carried out in section \ref{s-short-time}.

In section \ref{s-gradient}, we prove a simple estimate for the gradient $|Df|$ of solutions of the problem \eqref{Pe}. In section \ref{s-convexity}, we prove a crucial result that the time-section $\Om_t$ remains convex and the function $f(.,t)$ remains concave on $\Om_t$ for all time $t$. Convexity of $\Om_t$ guarantees that the free-boundary $\partial\Om_t$ does not touch itself and also enables us to prove the non-degeneracy of $|Df|$ near the free-boundary.

In section \ref{s-regularity}, we obtain an estimate for higher derivatives of $f$ in a neighborhood the free-boundary $\partial\Om_t$, uniformly in time $t$ and especially, in $\epsilon$, using the non-degeneracy of $|Df|$. This fact and the convexity guarantee that singular cannot develop on the free-boundary. The uniqueness for this regularized problem is obtained in section \ref{s-comparison}. In section \ref{s-long-time}, we then obtain a long-time existence result for solution of the regularized problem. Passing $\ep$ to 0, we then obtain a solution to the degenerate problem in section \ref{s-existence}. The uniqueness for the degenerate problem is then shown in section \ref{s-uniqueness}. In the last section, we show that solution to our degenerate problem vanishes in finite time.

\textbf{Acknowledgement.} I express my gratitude to my thesis advisor, P. Daskalopoulos, for suggesting this problem, and for her invaluable advices and support during the completion of this work.
\section{Definition of Solution}\label{s-statement}
In this section, we will define precisely what we mean by solution of the problem (P). We start by introducing some notations. For any $0 < t_1 < t_2 < T$, define
\begin{align*}
\Om_{(t_1, t_2)} &= \Om \cap \{ t_1 < t < t_2 \}.
\end{align*}
First, we require that the free-boundary $\partial\Om_t$ is in $C^1$ and the function $f$ is in
\[
C(0,T; C^1(\closure{\Om_t})).
\]
The equation
\[
f_t = \lap_p f \quad\text{in $\Om$}
\]
is then defined in the sense that for any test function $\te$ in $C^\infty_0 (\Om)$ and for any $0 < t_1 < t_2 < T$,
\[
\int_{\Om_{(t_1, t_2)}} f \te_t\,d\x dt - \left.\int f\te\,d\x \right|^{\Om_{t_2}}_{\Om_{t_1}} = \int_{\Om_{(t_1,t_2)}} |Df|^{p-2}\,Df \cdot D\te\,d\x dt.
\]
The Cauchy-Dirichlet conditions $f = 0$ on $\partial\Om_t$ and $f(.,0) = f_0$ are understood in the pointwise sense
\begin{align*}
f(\x,t) \to 0 &\quad\text{as $\x \to \x_0 \in \partial\Om_t$},\\
f(\x,t) \to f_0(\x) &\quad\text{as $t \to 0$}.
\end{align*}

Finally, the Neumann's boundary condition $|Du| = 1$ is defined in the following classical sense
\[
f_\nu(\x_0,t) = \lim_{h \to 0^{+}} \frac{f(\x_0 + h \nu)}{h} = 1.
\]
where $\x_0$ is a point on the free-boundary $\partial\Om_t$ and $\nu$ is the spatial inward unit normal vector at $\x_0$ with regards to $\partial\Om_t$.

\section{Short-time Existence for Regularized Problem}\label{s-short-time}
In this section, we will prove that the regularized free-boundary problem admits a solution for some positive time. We do it by a change of coordinates technique that transforms the problem into a fixed-domain problem. This technique has been used by other authors for different problems before (see for example \cite{DaskalopoulosL2002}, \cite{DaskalopoulosH1998}). Note that concavity is not needed in this result.

\begin{lemma}\label{short-time}
Assume that $\Om_0$ is  $C^\infty$. The function $f_0$ is in $C^\infty(\closure{\Om_0})$ and positive in $\Om_0$. Furthermore, on the boundary $\partial\Om_0$, $f_0$ satisfies
\[
f_0 = 0 \quad\text{and}\quad |Df_0| = 1.
\]
Then there exists a smooth solution to the regularized problem \eqref{Pe} for some $T > 0$.
\end{lemma}

\begin{proof}
The argument in this proof works for any dimension, but due to the complexity of some computation involved, we will present the proof for the case $n=2$ only.

A word on notation used in this proof : we use bold-face letters $\x, \mathbf{y},...$ to denote points in Euclidean spaces while normal letters $x, y, z,...$ for real numbers, scalars or components of points in Euclidean spaces.

Denote by $\Ss$ the smooth surface $z = f_0(x,y)$, $(x,y) \in \closure{\Om_0}$. Let $T = (T_1, T_2, T_3)$ be a smooth vector field on $\closure{\Om_0}$ such that $T(x,y)$ is not a tangential vector to the surface $\Ss$ at the point $f_0(x,y)$. Since $|Df_0| = 1$ on the boundary $\partial\Om_0$, we can also choose $T$ to be parallel to the plane $z = 0$ in a small neighborhood of $\partial\Om_0$.

It is known that for some positive, small enough $\eta$, we can define a change of spatial coordinates
\[
\Phi : \Om_0 \times [-\eta, \eta] \to \R^3
\]
by the formula
\[
\left( \begin{matrix}
x\\
y\\
x
\end{matrix} \right)
=
\Phi
\left( \begin{matrix}
u\\
v\\
w
\end{matrix} \right)
=
f_0
\left( \begin{matrix}
u\\
v
\end{matrix} \right)
+
w\, T
\left( \begin{matrix}
u\\
v
\end{matrix} \right).
\]
The map $\Phi$ defines $x,y$ and $z$ as smooth functions of $u,v$ and $w$ with smooth inverses.

The graph of $(x,y,f(x,y,t))$, $(x,y) \in \Om_t$ is then transformed to $(u,v,g(u,v,t))$, $(u,v) \in \Om_0$ via this coordinates change for some uniquely-defined $g$ if the surface $z = f(x,y,t)$ is sufficiently close to $\Ss$ ($(x,y,f(x,y,t)) \in \Phi(\Om_0 \times [-\eta, \eta])$ for all $(x,y) \in \Om_t$). When $f$ evolves as a function of $(x,y)$, $g$ evolves as a function of $(u,v)$. Importantly, the domain of $g$ is fixed as $\Om_0$ due to our requirement that $T$ is parallel to the plane $z=0$ on $\partial\Om_0$.

We will compute the evolution equation and the boundary condition of $g$. Denote by $x_u$, $x_v$, $x_w$, $y_u$, $y_v$, $y_w$, $z_u$, $z_v$ and $z_w$ the partial derivatives of the functions $x(u,v,w)$, $y(u,v,w)$ and $z(u,v,w)$. Similarly we denote partial second derivatives of $x, y$ and $z$ by $x_{uu}, x_{uv},...$.

We begin with first derivatives. Since $x,y$ and $z$ are functions of $u,v$ and $w$, while $w = g(u,v,t)$ is a function of $u,v$ and $t$, we have
\begin{align*}
\left( \begin{matrix}
\D{x}{u}& \D{y}{u}\\
\\
\D{x}{v}& \D{y}{v}
\end{matrix} \right)
&=
\left( \begin{matrix}
x_u + x_w \D{w}{u}& y_u + y_w \D{w}{u}\\
\\
x_v + x_w \D{w}{v}& y_v + y_w \D{w}{v}
\end{matrix} \right)\\
&=
\left( \begin{matrix}
x_u + x_w g_u& y_u + y_w g_u\\
\\
x_v + x_w g_v& y_v + y_w g_v
\end{matrix} \right).
\end{align*}
We can compute the partial derivatives of $u(x,y,t)$ and $v(x,y,t)$ by
\begin{align}\label{Ux}
\left( \begin{matrix}
\D{u}{x}& \D{u}{y}\\
\\
\D{v}{x}& \D{v}{y}
\end{matrix} \right)
=
\left( \begin{matrix}
\D{x}{u}& \D{x}{v}\\
\\
\D{y}{u}& \D{y}{v}
\end{matrix} \right)^{-1}
&=
\frac{1}{D}
\left( \begin{matrix}
\D{y}{v}& -\D{x}{v}\\
\\
-\D{y}{u}& \D{x}{u}
\end{matrix} \right)\\
&=
\frac{1}{D}
\left( \begin{matrix}
y_v + y_w g_v& - x_v - x_w g_v\\
\\
- y_u - y_w g_u& x_u + x_w g_u
\end{matrix} \right)
\end{align}
where
\begin{equation*}
D = \D{x}{u}\D{y}{v} - \D{x}{v}\D{y}{u} = (x_u y_v - x_v y_u) + (x_w y_v - y_w x_v)g_u + (y_w x_u - x_w y_u)g_v.
\end{equation*}
We then have
\begin{equation}\label{z-first}
\left( \begin{matrix}
f_x\\
\\
f_y
\end{matrix} \right)
=
\left( \begin{matrix}
\D{z}{x}\\
\\
\D{z}{y}
\end{matrix} \right)
=
\left( \begin{matrix}
\D{u}{x}& \D{v}{x}\\
\\
\D{u}{y}& \D{v}{y}
\end{matrix} \right)
\left( \begin{matrix}
\D{z}{u}\\
\\
\D{z}{v}
\end{matrix} \right)
=
\frac{1}{D}
\left( \begin{matrix}
\D{y}{v}& -\D{y}{u}\\
\\
-\D{x}{v}& \D{x}{u}
\end{matrix} \right)
\left( \begin{matrix}
z_u + z_w g_u\\
\\
z_v + z_w g_v
\end{matrix} \right).
\end{equation}
Next we compute the second-order derivatives. First, we have partial second order derivatives of $x$ with regards to $u$ and $v$.
\begin{align*}
\Dd{x}{u} &= x_{uu} + 2 x_{uw}\D{w}{u} + x_{ww} \left( \D{w}{u} \right)^2 + x_w \Dd{w}{u}\\
&= x_{uu} + 2 x_{wu} g_u + x_w g_{uu}\\
\Dd{x}{v} &= x_{vv} + 2 x_{wv} g_v + x_w g_{vv}\\
\DD{x}{u}{v} &= x_{uv} + x_{wu} g_v + x_{wv} g_u + x_w g_{uv}
\end{align*}
and similar formulae for $y$ and $z$.

Differentiate \eqref{z-first} we have
\begin{equation}\label{Zxx}
f_{xx} = \Dd{z}{x} = \D{z}{u}\,\Dd{u}{x} + \D{z}{v}\,\Dd{v}{x} + \Dd{z}{u}\left( \D{u}{x} \right)^2 + 2\DD{z}{u}{v}\,\D{u}{x}\,\D{v}{x}
 + \Dd{z}{v}\left( \D{v}{x} \right)^2.
\end{equation}

We need to compute second order derivatives of $u$ and $v$ with regards to $x$ and $y$. The formula \eqref{Zxx} is true if we substitute any function of $u$ and $v$ in place of $z$. Because second order derivatives of $x$ and $y$ with regards to $x$ are zero
\begin{align*}
0 &= \D{x}{u}\,\Dd{u}{x} + \D{x}{v}\,\Dd{v}{x} + \Dd{x}{u}\left( \D{u}{x} \right)^2 + 2\DD{x}{u}{v}\,\D{u}{x}\,\D{v}{x}
 + \Dd{x}{v}\left( \D{v}{x} \right)^2\\
0 &= \D{y}{u}\,\Dd{u}{x} + \D{y}{v}\,\Dd{v}{x} + \Dd{y}{u}\left( \D{u}{x} \right)^2 + 2\DD{y}{u}{v}\,\D{u}{x}\,\D{v}{x}
 + \Dd{y}{v}\left( \D{v}{x} \right)^2.
 \end{align*}
In other words
\begin{equation*}
\left( \begin{matrix}
\D{x}{u}& \D{x}{v}\\
\\
\D{y}{u}& \D{y}{v}
\end{matrix} \right)
\left( \begin{matrix}
\Dd{u}{x}\\
\\
\Dd{v}{x}
\end{matrix} \right)
+
\left( \begin{matrix}
\Dd{x}{u}\left( \D{u}{x} \right)^2 + 2\DD{x}{u}{v}\,\D{u}{x}\,\D{v}{x} + \Dd{x}{v}\left( \D{v}{x} \right)^2\\
\\
\Dd{y}{u}\left( \D{u}{x} \right)^2 + 2\DD{y}{u}{v}\,\D{u}{x}\,\D{v}{x} + \Dd{y}{v}\left( \D{v}{x} \right)^2
\end{matrix} \right)
= 0
\end{equation*}
or
\begin{align*}
\left( \begin{matrix}
\Dd{u}{x}\\
\\
\Dd{v}{x}
\end{matrix} \right)
&=
-
\left( \begin{matrix}
\D{x}{u}& \D{x}{v}\\
\\
\D{y}{u}& \D{y}{v}
\end{matrix} \right)^{-1}
\left( \begin{matrix}
\Dd{x}{u}\left( \D{u}{x} \right)^2 + 2\DD{x}{u}{v}\,\D{u}{x}\,\D{v}{x} + \Dd{x}{v}\left( \D{v}{x} \right)^2\\
\\
\Dd{y}{u}\left( \D{u}{x} \right)^2 + 2\DD{y}{u}{v}\,\D{u}{x}\,\D{v}{x} + \Dd{y}{v}\left( \D{v}{x} \right)^2
\end{matrix} \right)
\\
&=
-
\left( \begin{matrix}
\D{u}{x}& \D{u}{y}\\
\\
\D{v}{x}& \D{v}{y}
\end{matrix} \right)
\left( \begin{matrix}
\Dd{x}{u}\left( \D{u}{x} \right)^2 + 2\DD{x}{u}{v}\,\D{u}{x}\,\D{v}{x} + \Dd{x}{v}\left( \D{v}{x} \right)^2\\
\\
\Dd{y}{u}\left( \D{u}{x} \right)^2 + 2\DD{y}{u}{v}\,\D{u}{x}\,\D{v}{x} + \Dd{y}{v}\left( \D{v}{x} \right)^2
\end{matrix} \right).
\end{align*}

We then have
\begin{align*}
&\D{z}{u}\,\Dd{u}{x} + \D{z}{v}\,\Dd{v}{x} =
\left( \begin{matrix}
\D{z}{u} &\D{z}{v}
\end{matrix} \right)
\left( \begin{matrix}
\Dd{u}{x}\\
\\
\Dd{v}{x}
\end{matrix} \right)\\
&=
-
\left( \begin{matrix}
\D{z}{u} &\D{z}{v}
\end{matrix} \right)
\left( \begin{matrix}
\D{u}{x}& \D{u}{y}\\
\\
\D{v}{x}& \D{v}{y}
\end{matrix} \right)
\left( \begin{matrix}
\Dd{x}{u}\left( \D{u}{x} \right)^2 + 2\DD{x}{u}{v}\,\D{u}{x}\,\D{v}{x} + \Dd{x}{v}\left( \D{v}{x} \right)^2\\
\\
\Dd{y}{u}\left( \D{u}{x} \right)^2 + 2\DD{y}{u}{v}\,\D{u}{x}\,\D{v}{x} + \Dd{y}{v}\left( \D{v}{x} \right)^2
\end{matrix} \right)\\
&=
-
\left( \begin{matrix}
f_x & f_y
\end{matrix} \right)
\left( \begin{matrix}
\Dd{x}{u}\left( \D{u}{x} \right)^2 + 2\DD{x}{u}{v}\,\D{u}{x}\,\D{v}{x} + \Dd{x}{v}\left( \D{v}{x} \right)^2\\
\\
\Dd{y}{u}\left( \D{u}{x} \right)^2 + 2\DD{y}{u}{v}\,\D{u}{x}\,\D{v}{x} + \Dd{y}{v}\left( \D{v}{x} \right)^2
\end{matrix} \right).
\end{align*}
Substitute into \eqref{Zxx}
\begin{align*}
f_{xx} &= \left( \Dd{z}{u} - f_x \Dd{x}{u} - f_y \Dd{y}{u} \right) \left( \D{u}{x} \right)^2
+ \left( \Dd{z}{v} - f_x \Dd{x}{v} - f_y \Dd{y}{v} \right) \left( \D{v}{x} \right)^2\\
&\qquad {} + 2 \left( \DD{z}{u}{v} - f_x \DD{x}{u}{v} - f_y \DD{y}{u}{v} \right) \D{u}{x} \D{v}{x}.
\end{align*}
Let
\begin{align*}
E &= z_w - f_x x_w - f_y y_w\\
A &= \Dd{z}{u} - f_x \Dd{x}{u} - f_y \Dd{y}{u}\\
&= E g_{uu} + 2(z_{wu} - f_x x_{wu} - f_y y_{wu}) g_u + (z_{uu} - f_x x_{uu} f_y y_{uu})\\
B &= \Dd{z}{v} - f_x \Dd{x}{v} - f_y \Dd{y}{v}\\
&= E g_{vv} + 2(z_{wv} - f_x x_{wv} - f_y y_{wv}) g_v + (z_{vv} - f_x x_{vv} f_y y_{vv})\\
C &= \DD{z}{u}{v} - f_x \DD{x}{u}{v} - f_y \DD{y}{u}{v}\\
&= E g_{uv} + (z_{wv} - f_x x_{wv} - f_y y_{wv})g_u + (z_{wu} - f_x x_{wu} - f_y y_{wu})g_v\\
&\qquad {} + (z_{uv} - f_x x_{uv} - f_y y_{uv}),\\
\end{align*}
then
\begin{align*}
f_{xx} &= A \left( \D{u}{x} \right)^2 + B \left( \D{v}{x} \right)^2 + 2 C \D{u}{x} \D{v}{x}\\
&= E
\left(
\left( \D{u}{x} \right)^2 g_{uu} + \left( \D{v}{x} \right)^2 g_{vv} + 2 \D{u}{x} \D{v}{x} g_{uv}
\right) + \frac{F}{D^2}
\end{align*}
where $F$ is a smooth function of $u,v, g, g_u, g_w$. We have similar formulae for $f_{xy}$ and $f_{yy}$
\begin{align*}
f_{yy} &= E
\left(
\left( \D{u}{y} \right)^2 g_{uu} + \left( \D{v}{y} \right)^2 g_{vv} + 2 \D{u}{y} \D{v}{y} g_{uv}
\right) + \frac{F}{D^2}\\
f_{xy} &= E
\left(
\D{u}{x}\D{u}{y} g_{uu} + \D{v}{x}\D{v}{y} g_{vv} + \left(\D{u}{x}\D{v}{y} + \D{v}{x} \D{u}{y}\right) g_{uv}
\right) + \frac{F}{D^2}
\end{align*}
where $F$ denotes different smooth functions of $(u,v,g,g_u, g_v)$.

To compute $f_t$, we differentiate $z = f(x,y,t)$
\begin{align*}
&z_w\D{w}{t} = f_t + (f_x x_w + f_y y_w)w_t\\
&f_t = (z_w - f_x x_w - f_y y_w) g_t = E g_t.
\end{align*}

Substituting into the equation for $f$
\[
f_t = (|Df|^2 + \ep)^{q-1} \lap f + (p-2)(|Df|^2 + \ep)^{q-2} ( f_{xx} f_x^2 + f_{yy} f_y^2 + 2 f_{xy} f_x f_y )
\]
and simplifying $E$ from both sides we then obtain an evolution equation for $g$ in the form
\[
g_t = A^{ij}(u,v,g,Dg) g_{ij} + B(u,v,g,Dg).
\]
On the other hand, the boundary condition $|Df| = 1$ becomes
\[
C(u,v,g,Dg) = 0
\]
for some function $C$.

We claim the following is true when $g \equiv 0$ (i.e at $t=0$) :
\begin{itemize}
\item $A^{ij}, B$ and $C$ are smooth functions of $u,v,g$ and $Dg$.
\item $(A^{ij})$ is positive definite.
\item $C$ is oblique.
\end{itemize}

Because the surface $\Ss$ and the vector field $T$ are both smooth, it is clear that $A^{ij}, B$ and $C$ are smooth functions of $u,v,g$ and $Dg$ whenever
\begin{align*}
D &= \D{x}{u}\D{y}{v} - \D{x}{v}\D{y}{u} \neq 0\\
E &= z_w - f_x x_w - f_y y_w \neq 0.
\end{align*}
The condition that $E \neq 0$ follows from our choice that $T$ is transverse to $\Ss$. The condition $D \neq 0$ is a consequence of the fact that the function $\Phi$ is invertible in a neighborhood of $\Ss$.

Next, to show that $(A^{ij})$ is positive definite, we write
\[
A^{ij} = (|Df|^2 + \ep)^{q-1} A^{ij}_1 + (p-2)(|Df|^2 + \ep)^{q-2} A^{ij}_2
\]
where $A_1^{ij}$ is the coefficient of $g_{ij}$ ($i,j \in \{u,v\}$) obtained from the transformation of $\lap f$ and $A_2^{ij}$ from $f_x^2 f_{xx} + f_y^2 f_{yy} + 2 f_x f_y f_{xy}$. We can compute explicitly
\begin{align*}
A_1 &=
\left( \begin{matrix}
\left( \D{u}{x} \right)^2 + \left( \D{u}{y} \right)^2 &
\D{u}{x} \D{v}{x} + \D{u}{x} \D{v}{y}\\
\\
\D{u}{x} \D{v}{x} + \D{u}{x} \D{v}{y} &
\left( \D{v}{x} \right)^2 + \left( \D{v}{y} \right)^2
\end{matrix} \right)\\
A_2 &=
\left( \begin{matrix}
\left( f_x \D{u}{x} + f_y \D{u}{y} \right)^2 &
\left( f_x \D{u}{x} + f_y \D{u}{y} \right) \left( f_x \D{v}{x} + f_y \D{v}{y} \right)\\
\\
\left( f_x \D{u}{x} + f_y \D{u}{y} \right) \left( f_x \D{v}{x} + f_y \D{v}{y} \right) &
\left( f_x \D{v}{x} + f_y \D{v}{y} \right)^2
\end{matrix} \right).
\end{align*}
It is obvious that $A_1$ and $A_2$ are non-negative definite. Furthermore, if $D \neq 0$, $A_1$ is actually positive definite ($\det(A_1) = D^2$). It then follows readily that $(A^{ij})$ is positive definite.

For the proof that $C$ is oblique, we refer to the Appendix of \cite{DaskalopoulosL2002}.

From the continuity, there must exist a positive number $\de$ such that those three claims are true for all $g$ that satisfies $|g|_{C^1(\Om_0)} < \de$. It is then a consequence of standard theory of quasilinear parabolic equation with oblique boundary condition (see for examples \cite{Lieberman1996}, Chapter 14) that there exists a solution $g$ up to a positive time $T$ to the problem.
\[
\begin{cases}
g_t = A^{ij} g_{ij} + B &\text{in $\Om_0 \times (0,T)$}\\
C(u,v,g,Dg) = 0 &\text{on $\partial \Om_0 \times (0,T)$}\\
g(.,0) = 0.
\end{cases}
\]
This solution is actually smooth up to the boundary for all $t \in [0,T)$ since $\Om_0$ is smooth and $C(u,v,g,Dg)$ is a smooth function of $(u,v,g,Dg)$. Choose a number $T'$ in $(0,T]$ such that $|g| < \eta$ on $\Om_0 \times (0,T')$. Reverting back to the original coordinates system we then obtain a solution to the regularized problem \eqref{Pe} up to $T'$. It is clear that the domain $\Om_t$ is smooth and the solution $f$ is smooth up to the free-boundary for all time $0 < t < T'$.
\end{proof}

\section{Gradient Estimate}\label{s-gradient}
\begin{lemma}\label{gradient}
Assume the same hypotheses as in the Lemma \ref{short-time}. Furthermore, $f_0$ satisfies $|Df_0| \leq 1$ on $\Om_0$. If $f$ is a solution of the problem \eqref{Pe}, then
\begin{equation*}
|Df(\x,t)| < 1
\end{equation*}
for all $(\x,t) \in \Omega$.
\end{lemma}

\begin{proof}
We will show an equivalent fact that
\[
f_{\la}(\x,t) < 1
\]
for any unit vector $\la$.

Let
\begin{equation*}
a^{ij}(Df) = (|Df|^2+\ep)^{q-1} \de_{ij} + 2(q-1)(|Df|^2 + \ep)^{q-2} f_i f_j
\end{equation*}
where $\de_{ij}$ is the Kronecker delta function. Recall that we define $q = p/2$ throughout this work. Then the evolution equation of $f$ can be written in non-divergent form as
\begin{align*}
f_t &= (|Df|^2 + \ep)^{q-1} \lap f + 2(q-1)(|Df|^2 + \ep)^{q-2} f_{ij} f_i f_j\\
&= a^{ij} f_{ij}.
\end{align*}
We compute the evolution equation for $f_\la$
\begin{align*}
f_{\la t} &= a^{ij} f_{\la ij} + (f_{ij}\,Da^{ij})Df_\la.
\end{align*}
Since this equation satisfies the Strong Maximum Principle, $f_\la$ must attain its maximum value on the parabolic boundary of $\Om$. Because $f_\la \leq 1$ on the parabolic boundary of $\Omega$, it then follows that
\[
f_\la < 1
\]
in $\Om$ for all unit vector $\la$.
\end{proof}

\begin{lemma}\label{f-nu-nu}
Assume the same as in the last lemma, then at any point $\x_0$ on the free-boundary $\partial\Om_t$
\begin{equation*}
f_{\nu\nu} (\x_0,t) < 0
\end{equation*}
where $\nu$ is the inward normal vector at $\x_0$ with regards to $\partial\Om_t$.
\end{lemma}

\begin{proof}
Apply Hopf's Lemma to the evolution equation for $f_\nu$ from the last lemma, observing that $f_\nu$ attains the maximum value of 1 at $(\x_0,t)$.
\end{proof}

\section{Convexity}\label{s-convexity}

In this section we will show that the time-section $\Om_t$ remains convex and the function $f(.,t)$ remains concave on $\Om_t$. Normally, for this kind of question, the main difficulty lies in showing that $\Om_t$ remains convex. The arguments for the case $p=2$ as in \cite{Petrosyan2001} or \cite{DaskalopoulosL2002} do not translate directly to the case $p > 2$. On the other hand, our argument here can be simplified to give a new and simple proof for the case $p=2$. The argument relies heavily on the Neumann boundary condition $|Df| = 1$.

\begin{lemma}\label{convexity}
Assume the same hypotheses as in the Lemma \ref{short-time}. Furthermore, assume that  $\Om_0$ is strictly convex and $f_0$ is strictly concave on $\Om_0$. If $f$ is a solution to the problem \eqref{Pe} up to some positive time $T$, then $\Om_t$ is strictly convex and $f(.,t)$ is strictly concave for all $t \in [0,T)$.
\end{lemma}

\begin{proof}
We will show that
\begin{equation*}
f_{\la\la} (\x,t) < 0
\end{equation*}
for any point $(\x,t) \in \closure{\Omega}$, any unit vector $\la$ and any $t \in [0,T']$ where $T'$ is any number strictly less than $T$. Clearly this implies that $\Om_t$ is strictly convex and $f$ is strictly concave for all $t \in[0,T)$.

First, we compute the evolution equation of $f_{\la\la}$,
\begin{align*}
f_t &= a^{ij} f_{ij}\\
f_{\la t} &= a^{ij} f_{\la ij} + (f_{ij} Da^{ij})\cdot Df_\la\\
f_{\la\la t} &= a^{ij} f_{\la\la ij} + 2 (f_{\la ij} Da^{ij})\cdot Df_{\la} + (f_{ij}Da^{ij})\cdot Df_{\la\la} + (f_{ij} (Da^{ij})_\la)\cdot Df_\la\\
&= a^{ij} f_{\la\la ij} + (f_{ij} Da^{ij})\cdot Df_{\la\la} + ( 2f_{\la ij} Da^{ij} + f_{ij} D(a^{ij})_\la )\cdot Df_\la
\end{align*}
Since $f$ is smooth for all $t \in (0,T)$, there exists a finite number $C(T')$ such that for any unit vector $\la$ and any point $(\x,t) \in \closure{\Om} \cap \{ 0 < t \leq T' \}$
\begin{equation*}
| 2f_{\la ij} Da^{ij} + f_{ij} D(a^{ij})_\la | < C.
\end{equation*}

Choose a smooth function $v_0$ on $\R^n$ such that
\begin{equation*}
\begin{cases}
0 < v_0  < - (f_0)_{\la\la} &\text{in $\Om_0$}\\
v_0 = 0 &\text{on $\partial\Om_0$}.
\end{cases}
\end{equation*}
Such $v_0$ exists because $f_0$ is strictly concave on $\Om_0$. Let $v$ be the solution of the Cauchy-Dirichlet problem
\begin{equation*}
\begin{cases}
v_t = a^{ij} v_{ij} + (f_{ij} Da^{ij})\cdot Dv - C v &\text{in $\Om$}\\
v = 0 &\text{on $\partial{\Om} \cap \{ 0 < t < T \}$ }\\
v(.,0) = v_0 &\text{in $\Om_0$}.
\end{cases}
\end{equation*}
Applying Strong Maximum Principle and Hopf's Lemma to $v$ we easily deduce that
\begin{equation*}
\begin{cases}
v > 0 &\text{in $\Om$}\\
|Dv| > 0 &\text{on $\partial\Om \cap \{ 0 < t < T \}$}.
\end{cases}
\end{equation*}

We are going to show that
\begin{equation}\label{v-convex}
v + f_{\la\la} < 0
\end{equation}
for all $t \in [0,T']$ and all unit vector $\la$. Assuming that it is not the case, i.e there exists some point $(\x',t')$ and some unit vector $\la'$ such that
\begin{equation*}
(v + f_{\la'\la'})(\x',t') = 0
\end{equation*}
and
\[
v + f_{\la\la} < 0
\]
for all $t < t'$ and all unit vector $\la$. In other words, $t'$ is the first time \eqref{v-convex} fails. We consider two cases, $(\x',t')$ is an interior point or a boundary point. But first, note that we have the evolution equation for $V = v + f_{\la'\la'}$
\begin{equation}\label{V-evolution}
V_t = a^{ij}V_{ij} + (f_{ij}Da^{ij})\cdot DV +  (2f_{{\la'} ij} Da^{ij} + f_{ij} D(a^{ij})_{\la'} )\cdot Df_{\la'} - Cv.
\end{equation}
If $(\x', t')$ is an interior point, then because it is a maximum point of $V$ in $\Om_{t'}$, we have
\begin{align*}
a^{ij} V_{ij} &\leq 0\\
DV &= 0.
\end{align*}
Substitute into \eqref{V-evolution} we have
\[
V_t \leq ( 2f_{\la' ij} Da^{ij} + f_{ij} D(a^{ij})_\la' )\cdot Df_\la' - Cv.
\]
Because at the point $(\x',t')$
\begin{align*}
v + f_{\la'\la'} &\geq v + f_{\la\la}\\
\intertext{or}
f_{\la'\la'} &\geq f_{\la\la}
\end{align*}
for any other unit vector $\la$, we have
\begin{equation*}
f_{\la\la'} = 0
\end{equation*}
for any $\la \bot \la'$. Hence,
\begin{align*}
V_t &\leq \left((2f_{\la ij} Da^{ij} + f_{ij} D(a^{ij})_\la) \cdot \lambda'\right) f_{\la'\la'} - Cv\\
&< -C f_{\la'\la'} - Cv \quad\text{(remember $f_{\la'\la'} = -v < 0$)}\\
&= 0
\end{align*}
which contradicts the assumption that $(\x',t')$ is the first time $V = 0$. So $(\x',t')$ cannot be an interior point.

If $\x'$ is on $\partial\Om_{t'}$. Again, denote by $\nu$ the inward normal unit vector to $\partial\Om_{t'}$ at $x'$. Then at this point we have from definition of $(\x',t')$ and $\la'$,
\begin{align*}
(v + f_{\la'\la'})_\nu &\leq 0\\
f_{\nu \la'\la'} &\leq -v_\nu < 0.
\end{align*}

We will show that on the other hand
\begin{equation}\label{f-nu-la-la}
f_{\nu\la'\la'} = 0.
\end{equation}

We have from the Lemma \ref{f-nu-nu} that
\[
f_{\nu\nu} < 0.
\]
We also have as a consequence of the fact that $|Df| = 1$ on the free-boundary and $|Df| < 1$ in the interior that $f_{\nu\lambda} = 0$ for any tangential unit vector $\la$. Hence as a consequence of the fact $f_{\la'\la'} = 0$, $\la'$ must be a tangential vector of $\partial\Om_{t'}$. Otherwise, there would be a tangential vector $\la$ that lies on the same plane with $\nu$ and $\la'$ such that
\[
f_{\la\la} > 0
\]
which contradicts our assumption on $(\x',t')$ and $\la'$.

Without loss of generality, we can assume that $\nu = e_1$ and $\la' = e_2$. Because $e_1$ is the unit normal vector of $\partial\Om_{t'}$ at $x'$, in a small neighborhood of $\x'$, we can write $\partial\Om_{t'}$ as the graph of a smooth function
\[
x_1 = \ga(x_2, x')
\]
where $x' = (x_3,...x_n)$. From here to the end of the proof, we will use $\ga'$ and $\ga''$ to denote the first and second derivatives of $\ga$ with regards to $x_2$. Differentiate $f = 0$ with regards to $e_2$ we have
\begin{gather*}
f_1 \ga' + f_2 = 0\\
\intertext{or $\ga' = 0$ since $f_1 = 1$, $f_2 = 0$. Differentiate one more time and disregard all terms containing $\ga'$ we have}
f_1 \ga'' + f_{22} = 0\\
\intertext{and so $\ga'' = 0$ since $f_{22} = 0$ due to our assumption. Differentiate $Df\cdot Df = 1$ twice with regards to $e_2$ and disregard all terms containing $\ga'$ or $\ga''$ we obtain}
Df \cdot Df_{22} + Df_2 \cdot Df_2 = 0\\
\intertext{or}
f_{122} + |Df_2|^2 = 0.
\end{gather*}
As above, because $f_{ii} \leq 0 = f_{22} \;\forall\, i$, we have
\[
f_{2 i} = 0
\]
for all $i \neq 2$. But $f_{22} = 0$ as well, so $Df_2 = 0$. Hence
\[
f_{122} = 0
\]
which is exactly what we want to show in \eqref{f-nu-la-la}. We then have a contradiction. In other words
\[
f_{\la\la} < 0
\]
for all $(x,t) \in \closure{\Om} \cap \{ 0 < t < T \}$ and all unit vector $\la$ or equivalently, $\Om_t$ is strictly convex and $f(.,t)$ is strictly concave in $\Om_t$ for all $t \in [0,T)$.
\end{proof}

\section{Regularity near the Free-Boundary}\label{s-regularity}
In this section, we show that the degeneracy $|Df| = 0$ is kept away from the free-boundary. Consequently, the free-boundary is smooth, uniformly in $\ep$. It enables us to show that the limiting function obtained by letting $\ep$ go to 0 satisfies the boundary condition of the original problem. The proof depends crucially on the concavity of $f$.

We introduce some notations. We denote by $B_r(\x)$ the disk of radius $r$ around $x$
\[
B_r(\x) = \{ \y \in \R^n \st |\y - \x| < r \}
\]
when $\x \in \R^n$ and $r \in \R$. We write $B_r$ for $B_r(0)$. We also define
\[
A_r = \{ (\x,t) \st \dist(\x, \partial\Om_t) < r \} \cap \Om.
\]
For any point $\x = (x_1, x_2,...x_n)$, we define
\begin{align*}
\psi_1(\x) &= x_1\\
\psi'(\x) &= (x_2,...,x_n).
\end{align*}
\begin{lemma}\label{regularity}
Assume all hypotheses as in the Lemma \ref{convexity}. Assume also that there exist positive numbers $r, R$ and $m$ and a point $\x_0$ such that
\begin{align}\label{non-vanishing}
&B_r(\x_0) \subset \Om_t \subset B_R(\x_0) \quad\text{for all $t\in [0,T)$}\\
&f(\x,t) > m \quad\text{for all $(\x,t) \in B_r(\x_0) \times [0,T)$}.
\end{align}
Then for any $0 < T_1 < T$ and $k \in \Z^+$, there exist positive numbers $d(r, R, m, T_1)$ and $C(d,k)$ such that
\begin{equation*}
|f(.,t)|_{C^k(A_d \cap \{ T_1 \leq t < T \})} < C.
\end{equation*}
\end{lemma}

\begin{proof}
To simplify the notation, we assume that the conditions \eqref{non-vanishing} holds for $\x_0 = 0$. In other words
\begin{align*}
&B_r \subset \Om_t \subset B_R \quad\text{for all $t \in [0,T)$}\\
&f(\x,t) > m \quad\text{for all $(\x,t) \in B_r \times [0,T)$}.
\end{align*}

Let $(P,t)$ be a point on $\partial\Om$ for some $t \in [T_1, T)$. Fix this value of $t$ from here until the end of this proof. Without loss of generality, we can assume that
\begin{align*}
\psi_1(P) &< 0\\
\psi'(P) &= 0
\end{align*}

First, we will show that $f_1(\x,t)$ is bounded away from 0 in a neighborhood of $P$ in $\Om_t$. Consider any point $Q$ in $\Om_t$ that satisfies the following conditions
\begin{align*}
&\psi_1(Q) < 0\\
&|\psi'(Q)| < r\\
&f(Q,t) < m/2.
\end{align*}
Let $R = (0, \psi'(Q))$. Because $R \in B_r$, we have $f(R,t) > m$. We also have $f(Q,t) < m/2$ and $f_1(\x,t)$ decreases in $x_1$ as a consequence of concavity (here $f_1$ denotes the first derivative of $f$ with regards to $x_1$). Thus,
\begin{align*}
m/2 &< f(R,t) - f(Q,t)\\
&=\int_{\psi_1(Q)}^0 f_1((x_1, \psi'(Q)),t)\,dx_1\\
&\leq \abs{\psi_1(Q)} f_1(Q,t)\\
&\leq R f_1(Q,t)\\
f_1(Q,t) &\geq m/2R
\end{align*}

We just showed that if $\x$ satisfies
\begin{align*}
&\psi_1(\x) < 0\\
&|\psi'(\x)| < r\\
&f(\x,t) < m/2
\end{align*}
then
\[
f_1(\x,t) \geq m/2R.
\]
On the set containing all such $\x$, the Implicit Function Theorem says that there exists a function $g$  defined on the set
\begin{equation*}
\BB = \{ (y,x') \in \R\times\R^{n-1} \st 0 \leq y < m/2, |x'| < r \} \times [0,T)
\end{equation*}
such that
\begin{equation*}
\psi_1(\x) = g(f(\x,t), \psi'(\x)).
\end{equation*}

We will compute explicitly the evolution equation and boundary condition of $g$
\begin{gather*}
f_1 = \frac{1}{g_1}\\
f_i = -\frac{g_i}{g_1}\\
f_t = -\frac{g_t}{g_1}\\
f_{11} = -\frac{g_{11}}{g_1^3}\\
f_{1i} = -\frac{g_1 g_{1i} - g_i g_{11}}{g_1^3}\\
f_{ij} = -\frac{g_1^2 g_{ij} - g_1 g_j g_{1i} - g_1 g_i g_{1j} + g_i g_j g_{11}}{g_1^3}.
\end{gather*}

The boundary condition $|Df| = 1$ on $\partial\Om_t$ is equivalent to
\begin{equation*}
g_1 = \left( 1 + \sum_{i=2}^n g_i^2 \right)^{1/2}
\end{equation*}
on
\begin{equation*}
\{ (y,x') \in \R\times\R^{n-1} \st y = 0, |x'| < r \} \times [0,T).
\end{equation*}
Next we compute the evolution for $g$ on \BB. In all $\sum$ appearing in the following computations, unless explicitly marked otherwise, indices $i$ and $j$ run from 2 to $n$. Let
\[
M = 1 + \sum g_i^2.
\]
Then
\begin{align*}
|Df|^2 &= \frac{M}{g_1^2}\\
\lap f &= -\frac{1}{g_1^3}\,\left( g_{11} + g_1^2 \sum g_{ii} - 2g_1\sum g_i g_{1i} + g_{11} \sum g_i^2 \right)\\
&= -\frac{1}{g_1^3}\,\left( M g_{11} + g_1^2 \sum g_{ii} - 2 g_1 \sum g_i g_{1i} \right)\\
f_i f_j f_{ij} &= -\frac{1}{g_1^5}\,\left( g_{11} - 2 \sum g_i(g_1 g_{1i} - g_i g_{11}) \right.\\
&\qquad\qquad \left. {} + \sum g_i g_j (g_1^2 g_{ij} - g_1 g_i g_{1j} - g_1 g_j g_{1i} + g_i g_j g_{11}) \right)\\
&=-\frac{1}{g_1^5} \left( g_{11}M^2 - 2 g_1 M \sum g_i g_{1i} + g_1^2 \sum g_i g_j g_{ij} \right).
\end{align*}
Substitute into the equation for $f_t$,
\begin{align*}
g_t &= \frac{(M+\ep)^{q-1}}{g_1^{2q}}\,\left(M g_{11} + g_1^2 \sum g_{ii} - 2 g_1 \sum g_i g_{1i} \right)\\
&\qquad\qquad {} + 2(q-1)\frac{(M+\ep)^{q-2}}{g_1^{2q}}\,\left(g_{11}M^2 - 2 g_1 M \sum g_i g_{1i} + g_1^2 \sum g_i g_j g_{ij} \right)\\
&= b^{ij} g_{ij}.
\end{align*}
We want to show that there exist positive numbers $\la$, $\La$, independent of $\ep$ and $t$ such that
\[
\la \leq b^{ij}\xi_i \xi_j \leq \La
\]
for all unit vector $\xi\in\R^n$. First, because
\[
1 \geq |Df| \geq f_1 \geq m/2R,
\]
we have
\[
1 \leq M^{1/2} \leq g_1 \leq 2R/m.
\]
The upper bound $\La$ then is obvious. For the lower bound, because
\begin{align*}
M^2\xi_1^2 - 2 g_1 M\sum g_i \xi_1 \xi_i + g_1^2 \sum g_i g_j \xi_i \xi_j = (M\xi_1 - \sum g_i \xi_i)^2,
\end{align*}
it is enough to show that
\[
M\xi_1^2 - 2 g_1 \sum g_i \xi_1 \xi_i + g_1^2 \sum \xi_i^2 \geq \la
\]
for some positive $\la$. We have
\begin{align*}
\left(g_1^2 - \frac{1}{2n}\right)\left(g_i^2 + \frac{1}{2n}\right) - (g_1 g_i)^2 = \frac{1}{2n}\left( g_1^2 - g_i^2 - \frac{1}{2n} \right) \geq 0
\end{align*}
since $g_1^2 \geq M = 1 + \sum g_i^2$ and so
\begin{align*}
\left(\frac{1}{2n} + g_i^2\right)\xi_1^2 - 2 g_1 g_i \xi_1 \xi_i + (g_1^2 - \frac{1}{2n})\xi_i^2 \geq 0 \quad\text{for all } 2 \leq i \leq n.
\end{align*}
Summing up we obtain
\begin{align*}
\left( \frac{n-1}{2n} + \sum g_i^2 \right) \xi_1^2 - 2g_1\sum g_i \xi_1\xi_i + \left(g_1^2 - \frac{1}{2n}\right)\sum \xi_i^2 &\geq 0\\
M\xi^2 - 2 g_1 \sum g_i \xi_1 \xi_i + g_1^2 \sum \xi_i^2 &\geq \frac{1}{2n}.
\end{align*}

From the theory of quasi-linear parabolic equation with oblique boundary condition we can choose $d < \min(r, m/2)$ such that $g(.,t)$ is in $C^\infty$ on the set
\[
\{ (y, x') \in \R\times\R^{n-1} \st 0 \leq y \leq d, |x'| \leq d \} \times [T_1,T)
\]
and for any $k$, the norm $|g(.,t)|_{C^k}$ depends only on $k,d,r,R,m$ and $T_1$, not on $\ep$, $t$ or $g_0$. Revert back to $f$, we conclude that $f$ is smooth on the set
\[
\{ (\x,t) \in\Om \st |\psi'(\x)| \leq d, \psi_1(\x) < 0, f(\x,t) \leq d \} \cap \{ T_1 \leq t < T \}
\]
and again, for any $k$, the norm $|f(.,t)|_{C^k}$ on this set depends only on $k,d, r,R,m$ and $T_1$. Note that the above set includes the set
\[
B(P,d) \cap \Om_t.
\]
The conclusion is of course true for any point on $\partial\Om_t$ in place of $P$ where $t \in [T_1, T)$. The lemma then follows.
\end{proof}

\section{Comparison Principle}\label{s-comparison}
In the first lemma here, we show that if $f'_0$ is strictly greater than $f_0$, then a solution to the problem \eqref{Pe} with initial value $f'_0$ remains strictly greater than a solution with initial value $f_0$.
\begin{lemma}\label{uniqueness-regularized}
Suppose that $f$ and $f'$ are solutions up to some finite time $T$ to the problem \ref{Pe} and \Pee\ respectively for some $\ep \geq \ep' > 0$. Suppose also that at  the time $t=0$,
\begin{gather*}
\closure{\Om_0} \subset \Om'_0\\
f_0(\x) < f'_0(\x) \quad\text{in }\closure{\Om_0}.
\end{gather*}
Then for any $t \in (0,T)$,
\begin{gather*}
\closure{\Om_t} \subset \Om'_t\\
f(\x,t) < f'(\x,t) \quad\text{in }\closure{\Om_t}.
\end{gather*}
\end{lemma}

\begin{proof}
Since
\[
f_0(\x) < f'_0(\x) \quad\text{in }\closure{\Om_0}
\]
we can choose a positive number $m$ such that
\[
f_0(\x) + m < f'_0(\x) \quad\text{in }\closure{\Om_0}.
\]
Choose a positive number $\de$ such that $\de T < m$. We will show that
\begin{gather*}
\closure{\Om_t} \subset \Om'_t\\
f'(\x,t) - f(\x,t) - m + \de t > 0 \quad\text{in }\closure{\Om_t}
\end{gather*}
for all $t\in [0,T)$. Assuming it is not the case, there must be a first time $t_0$ such that at least one of the two above conditions is violated. Assume that the first condition is violated at $t_0$. In other words, $\partial\Om_t$ and $\partial\Om'_t$ touches at some point $\x_0$, then at that point
\[
f'(\x_0,t_0) - f(\x_0,t_0) - m + \de t_0  = -m + \de t_0 < 0
\]
which implies that the second condition must be violated before time $t_0$, contradicting our choice of $t_0$. Hence
\[
\closure{\Om_t} \subset \Om'_t
\]
for all $t \in [0,t_0]$. The second condition is violated implies that there is a point $\x_0 \in \closure{\Om_{t_0}}$ such that
\[
f'(\x_0,t_0) - f(\x_0,t_0) - m + \de t_0 = 0.
\]

We consider the case $\x_0 \in \partial\Om_{t_0} \subset \Om'_{t_0}$ first. Let $\nu$ be the inward unit normal to $\partial\Om_{t_0}$ at $\x_0$. From the definition of $(\x_0,t_0)$ we must have
\begin{align*}
(f'(\x_0,t_0) - f(\x_0,t_0) - m + \de t_0)_\nu &\geq 0\\
f'_\nu(\x_0,t_0) - f_\nu(\x_0,t_0) \geq 0\\
f'_\nu(\x_0,t_0) \geq 1
\end{align*}
which contradicts the result of Lemma \ref{gradient}.

If $\x_0 \in \Om_{t_0} \subset \Om'_{t_0}$, then because it is an minimum point for $f' - f$ on $\Om_{t_0}$, we have
\begin{gather*}
Df'(\x_0,t_0) = Df(\x_0,t_0)\\
0 \geq \lap f'(\x_0,t_0) \geq \lap f(\x_0,t_0)\\
0 \geq f'_{\nu\nu}(\x_0,t_0) \geq f_{\nu\nu}(\x_0,t_0).
\end{gather*}
Plug into the equation for $f'_t$ and $f_t$, recalling that $\ep \geq \ep'$ we obtain
\begin{align*}
f'_t &= (|Df'|^2 + \ep')^{q-1}\lap f' + 2(q-1)(|Df'|^2 + \ep')^{q-2} |Df'|^2 f_{\nu\nu}\\
&\geq (|Df|^2 + \ep)^{q-1}\lap f + 2(q-1)(|Df'|^2 + \ep)^{q-2} |Df|^ f_{\nu\nu}\\
&= f_t.
\end{align*}
On the other hand, because $t_0$ is the first time $f' - f - m + \de t = 0$,
\[
f'_t - f_t + \de \leq 0.
\]
Again, we arrive a contradiction. In other words,
\begin{gather*}
\closure{\Om_t} \subset \Om'_t\\
f' - f - m + \de t > 0
\end{gather*}
for all $t \in [0,T)$. The Lemma then follows readily.
\end{proof}

\begin{remark}
If we let
\begin{align*}
m &\to \min \,\{ f'_0(\x) - f_0(\x) \st \x \in \closure{\Om_0} \,\}\\
\de &\to 0
\end{align*}
we actually prove that
\[
m(t) = \min \,\{ f'(\x,t) - f(\x,t) \st \x \in \closure{\Om_t} \,\}
\]
is a non-decreasing function.
\end{remark}

We prove a slightly improved version of the last lemma.
\begin{lemma}\label{comparison}
Suppose $f$ and $f'$ are solutions up to time $T$ to the problem \Pe\ and \Pee\ respectively for some $\ep \geq \ep' > 0$. Suppose also that at the time $t=0$,
\begin{gather*}
\Om_0 \subset \Om'_0\\
f_0(\x) \leq f'_0(\x) \quad\text{in }\Om_0.
\end{gather*}
Then for any $t \in [0,T)$,
\begin{gather*}
\Om_t \subset \Om'_t\\
f(\x,t) \leq f'(\x,t) \quad\text{in }\Om_t.
\end{gather*}
\end{lemma}

\begin{proof}
Without loss of generality, we can assume that $f_0$ attains its maximum value at the origin. For each positive $\la$, define
\begin{align*}
\Om^\la &= \{ (\x,t) \st (\la\x,\la^2 t) \in \Om \}\\
f^\la &= \frac{1}{\la}\,f(\la \x, \la^2 t)\\
\Om_0^\la &= \{ \x \st \la \x \in \Om_0 \}\\
f_0^\la &= \frac{1}{\la}\,f_0(\la \x, \la^2 t).
\end{align*}
It is clear that $f^\la$ is a solution to the problem \eqref{Pe} with respect to the initial data $f_0^\la$. Furthermore, since $f_0$ is concave, for each $\la > 1$ we have
\begin{gather*}
\closure{\Om_0^\la} \subset \Om_0 \subset \Om'_0\\
f_0^\la < f_0 \leq f'_0 \quad\text{in }\Om^\la_0.
\end{gather*}
The last lemma says that for all $t \in [0, T/\la^2)$,
\begin{gather*}
\closure{\Om_t^\la} \subset \Om'_t\\
\frac{1}{\la}\,f(\la\x, \la^2 t) = f^\la(\x,t) < f'(\x,t) \quad\text{in $\Om^\la_t$}.
\end{gather*}
Let $\la \to 1$ we obtain
\begin{gather*}
\Om_t \subset \Om'_t\\
f(\x,t) \leq f'(\x,t) \quad\text{in $\Om_t$}
\end{gather*}
for all $t \in [0,T)$.
\end{proof}

\begin{corollary}
Suppose $(f,\Om)$ and $(f',\Om')$ are two solutions to the problem \eqref{Pe} with respect to the same initial value $f_0$ up to time $T$. Then $f = f'$ on $\R^n \times [0,T)$.
\end{corollary}

\section{Long-Time Existence for the Regularized Problem}\label{s-long-time}
\begin{lemma}\label{long-time}
Assume that $f_0$ satisfies all hypotheses of the Lemma \ref{short-time}. Then there exists a unique solution to the problem \ref{Pe} up to some positive time $T > 0$ where
\begin{equation*}
\lim_{t\to T} f(\x,t) = 0
\end{equation*}
for all $\x \in \R^n$.
\end{lemma}

\begin{proof}
Let $T$ be the maximal existence time for solutions to the problem \ref{Pe} with the initial data $f_0$. Due to the uniqueness result in section \ref{s-uniqueness}, there must be a solution $f$ that exists up to time $T$. From the short-time existence result, $T$ must be positive. We will show that
\[
\lim_{t \to T} f(\x,t) = 0 \quad\text{for all } \x \in \R^n.
\]
Assuming otherwise, then the same argument in the proof for the Lemma \ref{finite} can be used to show that $T$ must be finite. In other words, due to the concavity of $f_0$, there exists a number $c < 0$ such that
\[
\divg \left( (|Df_0|^2 + \ep)^{q-1}\,Df_0 \right) < c \quad\text{in } \Om_0
\]
and consequently,
\[
T \leq \frac{\max f_0}{\abs{c}}.
\]
We will prove that we can then extend this solution to a time $T' > T$. From the concavity of $f(.,t)$, $f$ is a decreasing function in $t$. Define
\begin{equation*}
f_T(\x) = \lim_{t \to T} f(\x,t).
\end{equation*}
Since $|Df| \leq 1$ in $\Om$, $f_T$ is continuous. Because $f_T$ is not identically 0,
there exist a ball $B_r(\x')$ and a positive number $m$ such that
\[
f_T > m \quad\text{in $B_r(\x')$}.
\]
From the Lemma \ref{regularity}, for all $t \in [T/2,T)$, there exists a positive number $d$ such that $f$ is smooth up to the boundary and time $T$ in the set
\[
\{ (\x, t) \st \dist(\x,\partial\Om_t) < d \} \cap \Om_{[T/2,T)}.
\]
Combine with the smoothness (depending on $\ep$) of $f$ up to time $T$ in the interior of $\Om_{[T/2,T)}$ from the standard theory of parabolic equation, we obtain the smoothness up to the boundary and time $T$ of $f$ in $\Om_{[T/2,T)}$. Consequently, $f_T$ is smooth up to the boundary. From the Lemma \ref{convexity}, we know that $\Om_T$ is convex and $f_T$ is concave in $\Om_T$. However, we need a stronger result that $\Om_T$ is strictly convex and $f_T$ is strictly concave in $\Om_T$ in order to apply the Lemma \ref{short-time}. In deed, we can improve the result in the lemma \ref{convexity} by duplicating the proof and substituting $T'$ by $T$ directly. In that proof, because we did not have the smoothness of $f$ up to time $T$, we need to introduce $T' < T$ to guarantee the existence of a finite number $C(T')$ such that
\begin{equation*}
| 2f_{\la ij} Da^{ij} + f_{ij} D(a^{ij})_\la | < C
\end{equation*}
for all $t \in [t,T']$. But now we have the smoothness of $f$ up to time $T$, we can derive the fact that there exists a number $C(T)$ such that the above inequality holds for all $t \in [0,T)$. The proof then guarantees that $f$ is strictly concave at the time $T$.

The function $f_T$ now satisfies all hypotheses of the Lemma \ref{short-time}. By that Lemma, we can then extend the solution $f$ to some time $T' > T$. It contradicts the maximality of $T$. So we must have
\[
\lim_{t \to T} f(\x,t) = 0
\]
for all $\x \in \R^n$.
\end{proof}

\section{Existence of Solution to the $p$-Laplacian problem}\label{s-existence}
In this section, we will pass $\ep$ to 0 and obtain a solution to our degenerate problem.
\begin{lemma}\label{existence}
Assume that $\Om_0$ is a bounded and convex domain. The function $f_0$ is positive and concave in $\Om_0$. Furthermore, on the boundary $\partial\Om_0$, $f_0$  satisfies
\begin{align*}
f_0(\x) &= 0 \quad\text{for all $\x$}\\
|Df_0(\x)| &= 1 \quad\text{for a.e. $\x$}.
\end{align*}
Then there exists a solution to the problem \eqref{P} up to some time $T$ where
\[
\lim_{t \to T} f(\x, t) = 0 \quad \forall \x \in \R^n.
\]
The free-boundary $\partial\Om_t$ is smooth for all $t \in (0,T)$.
\end{lemma}

\begin{proof}
Choose a sequence of functions $f^\ep_0$ with positive sets $\Om^\ep_0$ for all $\ep \in (0,1)$ such that
\begin{align*}
&\Om^\ep_0 \in C^\infty \text{ and } f^\ep_0 \in C^\infty(\closure{\Om^\ep_0}),\\
&\Om^\ep_0 \text{ is strictly convex},\\
&f^\ep_0 \text{ is strictly concave},\\
&\Om^{\ep_1}_0 \subset \Om^{\ep_2}_0 \text{ and } f^{\ep_1}_0 \leq f^{\ep_2}_0 \text{ if $\ep_1 > \ep_2$},\\
&\Om_0 = \cup \Om^\ep_0 \text{ and } f_0(\x) = \lim_{\ep \to 0} f^\ep_0(\x) \text{ for all } x\in\R^n,\\
&|Df^\ep| = 1 \text{ on } \partial\Om^\ep.
\end{align*}
In other words, $f^\ep_0$ is an increasing sequence of smooth and strictly concave function that converge to $f_0$ as $\ep \to 0$. From the Lemma \ref{long-time}, for each $\ep$, there exists a unique solution $f_\ep$ to the problem \eqref{Pe} up to some time $T^\ep$ where it vanishes identically. We will prove that $\lim_{\ep \to 0} f^\ep$ is a solution to the original problem \eqref{P}.

From the lemma \ref{comparison} and our choice of $f^\ep_0$, it is clear that if $\ep_1 > \ep_2$, then
\begin{align*}
&T^{\ep_1} \leq T^{\ep_2}\\
&\Om^{\ep_1} \subset \Om^{\ep_2}\\
&f^{\ep_1} \leq f^{\ep_2} \text{ in $\Om^{\ep_1}$}.
\end{align*}
Define
\begin{align*}
T &= \lim_{\ep \to 0} T^\ep\\
\Om &= \cup \Om^\ep\\
f(\x,t) &= \lim_{\ep \to 0} f^\ep(\x,t) \quad\text{for $(\x,t) \in \R^n \times [0,T)$}.
\end{align*}
Due to the uniform smoothness of $f^\ep$ in a neighborhood of $\partial\Om^\ep$, we have
\begin{gather*}
\Om_{(0,T)} \in C^\infty\\
f = 0 \quad\text{and}\quad |Df| = 1 \quad\text{on $\partial\Om \times \{ 0 < t < T \}$}.
\end{gather*}

If $(\x,t) \in \Om$, there exists an $\ep_0$ such that $(\x,t) \in \Om^\ep$ for all $\ep < \ep_0$. Since $f^\ep(\x,t)$ increases as $\ep$ decreases to 0,
\[
f(\x,t) = \lim_{\ep \to 0} f^\ep(\x,t) > 0.
\]
From the bound $|Df| \leq 1$ and the Corollary 2.15 in Chapter II of \cite{Lieberman1996}, we have interior $C_t^{0,1/2}$ estimate for $f^\ep$ as functions of $t$, uniformly in $t$ and $\ep$. Together with the fact that for every $\x$,
\[
\lim_{t \to T^\ep} f^\ep(\x,t) = 0
\]
we have
\[
\lim_{t \to T} f(\x,t) = 0.
\]

Because for every $\ep$ and $t$
\[
f^\ep(\x,t) \leq f^\ep(\x,0) < f_0(\x,0),
\]
we have
\[
f(\x,t) = \lim_{\ep \to 0} f^\ep(\x,t) \leq f_0(\x).
\]
On the other hand, since $f^\ep(\x,t)$ increases as $\ep$ decreases to 0,
\[
\lim_{t \to 0} f(\x,t) \geq \lim_{t \to 0} f^\ep(\x,t) = f^\ep_0(\x)
\]
for any $\ep$. Consequently,
\[
\lim_{t \to 0} f(\x,t) \geq \lim_{\ep \to 0} f_0^\ep(\x) = f_0(\x).
\]
Hence
\[
\lim_{t \to 0} f(\x,t) = f_0(\x).
\]

From $|Df^\ep| \leq 1$, we can choose a sequence of $\ep$ converging to $0$ such that
\[
Df^\ep \weaklyto Df
\]
in all compact subsets of $\Om$. Given any function $\te \in C^{\infty}_0(\Om)$ and $0 < t_1 < t_2 < T$ we have from the equation
\[
f^\ep_t = \divg ( (|Df^\ep|^2 + \ep)^{q-1}\,Df^\ep )
\]
that
\[
\int_{\Om_{(t_1, t_2)}} f^\ep \te_t\,d\x\,dt - \left. \int f^\ep\te\,d\x \right|^{\Om_{t_2}}_{\Om_{t_1}} =
\int_{\Om_{(t_1, t_2)}} (|Df^\ep|^2 + \ep)^{q-1}\,Df^\ep\cdot D\te\,d\x\,dt.
\]
Passing to the limit we then obtain
\[
\int_{\Om_{(t_1, t_2)}} f \te_t\,d\x\,dt - \left. \int f\te\,d\x \right|^{\Om_{t_2}}_{\Om_{t_1}} =
\int_{\Om_{(t_1, t_2)}} |Df|^{2(q-1)}\,Df\cdot D\te\,d\x\,dt.
\]
\end{proof}

\section{Uniqueness}\label{s-uniqueness}
\begin{lemma}
Solution obtained in the Lemma \ref{existence} is the unique solution to the problem (P).
\end{lemma}

\begin{proof}
Assume that there exists another solution $g$ to the problem (P). Let $\Om^*$ be the positive set of $g$ and $T^*$ its existence time. Also without loss of generality, assuming that $f_0$ attains it maximum value at $0$. For each positive $\la$, it is clear that
\[
g^{\la}(\x,t) = \la^{-1} g(\la^2 \x, \la^{p+2} t)
\]
is a solution to the problem \ref{Pe} with positive set
\[
\Om^\la = \{ (\x,t) \st (\la^2 \x, \la^{p+2} t) \in \Om^* \}
\]
and initial data
\[
g^\la_0(\x,t) = \la^{-1} f_0 ( \la^2 \x, \la^{p+2} t ).
\]

Clearly, for $\la < 1$,
\begin{gather*}
\closure{\Om_0} \subset \Om^\la_0\\
f_0 < g^\la_0 \quad\text{in $\Om_0$}.
\end{gather*}
We will show that for all $ t < \min(T, \la^{-(p+2)} T^*)$,
\begin{gather*}
\closure{\Om_t} \subset \Om^\la_t\\
f(\x,t) < g^\la(\x,t) \quad\text{in $\Om_t$}.
\end{gather*}
Assuming it is not the case, then there must be a first time $t_0$ where at least one of those two inequalities is violated. If the first one is violated at the time $t_0$, it means $\partial\Om_t$ and $\partial\Om^\la_t$ touch at some point $\x_0$. At that point $(\x_0,t_0)$,
\[
|Df| = 1 > \la = |Dg^\la|.
\]
There must be then a point $\x \in \Om_t$ such that
\[
f(\x,t_0) > g^\ep(\x,t_0).
\]
which implies that the second inequality must be violated at some time before $t_0$. So, up to time $t_0$,
\[
\closure{\Om_t} \subset \Om^\la_t.
\]
Consequently, on the parabolic boundary of $\Om_{[0,t_0]}$, $f < g^\la$. Thus, from the lemma 3.1 in Chapter VI of \cite{DiBenedetto1993}, we have $f < g$ in $\Om_{[0,t_0]}$ which contradicts our choice of $t_0$. Hence for all $t < \min(T,\la^{-(p+2)} T^*)$,
\begin{gather*}
\closure{\Om_t} \subset \Om_t^\la\\
f(\x,t) < g^\la(\x,t) = \la^{-1} g(\la^2 \x, \la^{p+2} t) \quad\text{in $\Om_t$}.
\end{gather*}
Let $\la \to 1$ we obtain for all $t < \min(T,T^*)$,
\begin{gather*}
\Om_t \subset \Om^*_t\\
f(\x,t) \leq g(\x,t) \quad\text{in $\Om_t$}.
\end{gather*}

Arguing similarly for $\la > 1$ we on the other hand obtain
\begin{gather*}
\Om^*_t \subset \Om_t\\
g(\x,t) \leq f(\x,t) \quad\text{in $\Om^*_t$}.
\end{gather*}

Thus for all $t < \min(T, T^*)$
\begin{gather*}
\Om_t = \Om^*_t\\
f(\x,t) = g(\x,t).
\end{gather*}

\end{proof}

\section{Vanishing in finite time}\label{s-finite}
\begin{lemma}\label{finite}
The existence time $T$ of the solution obtained in the Lemma \ref{existence} is finite.
\end{lemma}

\begin{proof}
Clearly from the Comparison Principle in section \ref{s-comparison} and scaling that it is enough to prove this lemma for one particular initial function $f_0$. We show that if a smooth function $f_0$ satisfies all hypotheses of the Lemma \ref{existence} and
\[
\lap_p f_0 < c
\]
for some $c < 0$, then for any $0 < t_1 < t_2 < T$ and any $\x \in \Om_{t_1}$, the corresponding solution $f$ satisfies the inequality
\begin{equation}\label{f_t-bound}
f(\x,t_2) - f(\x,t_1) \leq c ( t_2 - t_1 ).
\end{equation}
It then readily follows that
\[
T \leq \frac{\max f_0}{\abs{c}}.
\]

Choose the sequence $\{ f_0^\ep \}$ so that
\[
\divg ( (|Df_0^\ep|^2 + \ep)^{q-1}\,Df_0^\ep ) < c.
\]
We will show that $f^\ep$ satisfies
\[
f_t^\ep \leq c
\]
for all $\ep$ and \eqref{f_t-bound} then follows immediately.

Differentiating the equation satisfied by $f^\ep$ with respect to $t$, it is easy to see that $f^\ep_t$ satisfies the Maximum Principle. Since $f^\ep_t \leq c$ at the time $t = 0$ from our choice of $f_0^\ep$, all we need to show is that $f^\ep_t$ cannot attain its maximum value on the free-boundary. From the Hopf's Lemma, if $f_t^\ep$ attains its maximum value at a point $\x_0$ on the free-boundary $\partial\Om_{t_0}$, then we must have
\[
(f^\ep_\nu)_t (\x_0, t_0) = (f^\ep_t)_\nu (\x_0, t_0) < 0
\]
where $\nu$ is the inward unit normal at $\x_0$ with respect to $\partial\Om_{t_0}$. On the other hand, since $\Om_t$ shrinks in time, for any $t < t_0$ $\x_0 \in \Om_t$ and so,
\[
f^\ep_\nu (\x_0, t) < 1
\]
while
\[
f^\ep_\nu (\x_0,t_0) = 1
\]
which lead to
\[
(f^\ep_\nu)_t (\x_0, t_0) \geq 0.
\]
\end{proof}

\begin{remark}
The finiteness for the existence time holds for any initial data with bounded support, not just concave ones.
\end{remark}

\bibliographystyle{amsplain}
\bibliography{PDE}

\end{document}